\documentclass[reqno, 11pt]{amsart}
\usepackage{url}
\usepackage[colorlinks, linkcolor=blue, citecolor=blue, urlcolor=blue]{hyperref}
\usepackage{times}

\newtheorem{theorem}{Theorem}[section]
\newtheorem{lemma}[theorem]{Lemma}

\theoremstyle{definition}

\newtheorem{remark}[theorem]{Remark}

\numberwithin{equation}{section}

\usepackage{bbm}
\usepackage{url}
\usepackage{euscript}
\usepackage{amsmath}
\usepackage{amsthm}

\usepackage{amsxtra}
\usepackage{amssymb}
\usepackage{pifont}
\usepackage{amsbsy}

\usepackage{graphicx}
\usepackage{epstopdf}

\usepackage{tikz}
\usepackage{pgfplots}
\pgfplotsset{compat=1.18}
\usetikzlibrary{arrows.meta}

\usepackage[backend=biber,style=numeric,doi=true,url=false,eprint=true]{biblatex}
\DeclareFieldFormat{title}{\mkbibemph{#1}}
\DeclareFieldFormat[article]{title}{\mkbibemph{#1}}
\DeclareFieldFormat[misc]{title}{\mkbibemph{#1}}

\DeclareFieldFormat{date}{#1}
\DeclareFieldFormat{year}{#1}
\DeclareFieldFormat[article]{issue}{#1}

\renewbibmacro*{date}{%
  \printtext[parens]{\printdate}%
}

\renewbibmacro*{issue+date}{%
  \printtext[parens]{\printdate}%
  \newunit}

\newskip\aline \newskip\halfaline
\title[Half Strong Ill-Posedness of $2 \frac{1}{2}$D electron MHD with Fractional Resistivity]{Half Strong Ill-Posedness of $2 \frac{1}{2}$D Electron Magnetohydrodynamics with Fractional Resistivity}

\author{Xiaotong (Dawson) Yang}
\address{Xiaotong (Dawson) Yang: Department of Mathematics, Statistics, and Computer Science, University of Illinois Chicago, Chicago, IL 60607, USA}
\email{xyang212@uic.edu}
\urladdr{\url{https://homepages.math.uic.edu/~xyang212/}}

\author{Haoming Zhu}
\address{Haoming Zhu: Department of Mathematics, Statistics, and Computer Science, University of Illinois Chicago, Chicago, IL 60607, USA}
\email{hzhu54@uic.edu}
\urladdr{\url{https://sites.google.com/view/haomingzhu/home}}


\date{Last updated: \today}

\begin{document}

\begin{abstract}
We study the $2\frac{1}{2}$D electron magnetohydrodynamics (MHD): the electron MHD system that has $3$D magnetic field but is independent of $z$-variable. We establish a “half” strong ill-posedness result in $2\frac{1}{2}$D electron MHD with fractional resistivity $(-\Delta)^\alpha$ in the supercritical Sobolev space $H^{\beta}\times H^{\beta-1}$ for $3<\beta<4-2\alpha$. Specifically, we construct small initial data $(a_0,b_0)$ whose solution develops a norm inflation in $a$ but the norm of $b$ remains small.
\end{abstract}

\maketitle
  
\tableofcontents

\section{Introduction}
Consider the standard electron magnetohydrodynamics (MHD) system
\begin{equation}\label{eq:EMHD}
\begin{cases}
B_t +\nabla \times ((\nabla\times B)\times B) = \mu \Delta B, \\
\nabla \cdot B = 0
\end{cases}
\end{equation}
In physics, (\ref{eq:EMHD}) describes the evolution of the magnetic field in plasmas whose particle velocity is so small that it may be neglected. In fact, considering the fast changing rate of the magnetic field, the electron MHD system may be viewed as a good approximation to the real plasma. Moreover, the Hall term $\nabla \times ((\nabla\times B)\times B)$ helps physicists understand many interesting phenomena in magnetohydrodynamics, including magnetic reconnection (see, e.g. \cite{hallreconnection}). It is more singular than the Euler-type nonlinearity $B \cdot \nabla B$, thus making it more challenging to estimate mathematically.

If one further assumes the plasma is uniform in one direction ($z$-direction for example), like in a long tube or circulated tube (Tokamak), then the so called $2\frac{1}{2}$D electron MHD system comes into play. In fact, consider a planar magnetic potential $A=A(t,x,y)$ independent of $z$-variable and its magnetic field: 
\[
B = \nabla \times A(t,x,y) = (\frac{ \partial A_z}{\partial y},-\frac{\partial A_z}{\partial x}, \frac{\partial A_y}{\partial x}-\frac{\partial A_x}{\partial y}).
\]
Denote $a=-A_z$ and $b=\frac{\partial A_y}{\partial x}-\frac{\partial A_x}{\partial y}$. Note that $B:\mathbb R_t \times \mathbb R_{(x,y)}^2 \to \mathbb R^3$ has the following structure: $B=\nabla \times (ae_z) + b\,e_z$. Plugging this expression of $B$ in (\ref{eq:EMHD}), we derive the $2\frac{1}{2}$D electron MHD system:
\begin{equation}\label{eq:25d}
\begin{cases}
a_t + \nabla^\perp b\cdot \nabla a  = \mu\Delta a,\\[2pt]
b_t + \nabla^{\perp}a\cdot\nabla \Delta a = \mu\Delta b.
\end{cases}
\end{equation}
Considering $2\frac{1}{2}$D electron MHD as a reduced model of electron MHD and Hall MHD (where plasma particles' velocity is not small), its ill-posedness may contribute to understanding the ill-posedness of more general models. In this paper, we consider the generalized $2\frac{1}{2}$D electron MHD system with fractional magnetic resistivity:
\begin{equation}\label{eq:resistive-25d}
\begin{cases}
a_t + \nabla^\perp b\cdot \nabla a + \mu(-\Delta)^{\alpha} a = 0,\\[2pt]
b_t + \nabla^{\perp}a\cdot\nabla \Delta a + \nu(-\Delta)^{\alpha} b = 0,
\end{cases}
\end{equation}
where $\mathcal{F}((-\Delta)^\alpha f)(\xi)=|\xi|^{2\alpha}\mathcal{F}(f)(\xi)$.

Chae, Wan, and Wu \cite{Chae-Wan-Wu} showed that the Hall-MHD system is locally well-posed in $H^{1+d/2+}(\mathbb R^d)$ when there is no dissipation in velocity but a fractional magnetic resistivity $(-\Delta)^\alpha B$ with $\alpha>1/2$. (Here, we use $H^{1+d/2+}$ to denote the Sobolev spaces $H^\beta$ with $\beta>1+d/2$. For the rest of the paper it is similar.) Following it, there has been a surge of papers on fractionally resistive Hall-MHD, most of them on locally well-posedness. Later Dai \cite{daiLWP} studied Hall MHD with full dissipation in velocity and fractional magnetic resistivity, and proved local well posedness in $H^{2+d/2-2\alpha+}(\mathbb R^d)$ but still with the restriction $\alpha > 1/2$. Also, Zhang and Zhao \cite{ZhangZhao} looked into general Hall MHD with velocity dissipation $(-\Delta)^\beta u$ such that $0\le\beta\le1$, and proved global well-posedness when magnetic resistivity $\alpha = 1/2$ in $H^{5/2+}(\mathbb R^3)$ under small initial data. However, the local well-posedness of Hall or electron MHD when $\alpha = 1/2$ remains an open problem.

As for the ill-posedness of Hall/electron/$2\frac{1}{2}$D electron MHD, the main focus is on the inviscid case. In sequential works \cite{Jeong-Oh-2.5} \cite{Jeong-Oh-Ill} by Jeong and Oh, non-existence of solution to inviscid electron MHD in $H^{7/2+}(\mathbb R^3)$ is established. And later Dai \cite{dai2024} established ill-posedness of inviscid $2\frac{1}{2}$D electron MHD through a norm inflation in $H^{\beta}(\mathbb R^2)\times H^{\beta-1}(\mathbb R^2)$ (for $a$ and $b$ respectively) with $1<\beta<4$. For electron MHD with fractional resistivity, Jeong and Oh \cite{Jeong-Oh-2.5} established strong ill posedness of electron MHD in $H^{s}$ for $s\ge 3$ with $0\le \alpha < 1/2$.

Returning to the $2\frac{1}{2}$ electron MHD system with fractional dissipation (\ref{eq:resistive-25d}), note that the system is invariant under scaling:
\begin{align*}
    a(t,x) &\to \lambda^{2\alpha-3}a(\lambda^{2\alpha}t,\lambda x) \\
    b(t,x) & \to \lambda^{2\alpha-2}b(\lambda^{2\alpha }t, \lambda x)
\end{align*}
whose critical Sobolev spaces (i.e. norms are invariant under scaling) are $a\in \dot{H}^{4-2\alpha}(\mathbb{R}^2)$ and $b\in \dot{H}^{3-2\alpha}(\mathbb{R}^2)$. So we shall expect ill-posedness in supercritical space $\dot H^{4-2\alpha-}\times \dot H^{3-2\alpha-}$. In this paper, we establish the strong ill-posedness of (\ref{eq:resistive-25d}), inspired by Dai \cite{dai2024} and the work of Guo and Luo \cite{guo2025norminflationcriticalsqg} in SQG.

\begin{theorem}[Half strong ill-posedness in generalized $2\frac{1}{2}$D electron MHD]\label{thm:main}
Let $3<\beta<4-2\alpha$ and $0\le\alpha<1/2$, then for any small $\epsilon>0$, any large $\Lambda>0$, and any small time $t_*>0$, there exist smooth initial data $(a_0,b_0)$ such that,
\[
\|a_0\|_{H^\beta}+\|b_0\|_{H^{\beta-1}} < \epsilon,
\]
such that the generalized $2\frac{1}{2}$D electron MHD system \eqref{eq:resistive-25d} has a solution whose norm inflates
\[
\|a(t_*)\|_{\dot H^\beta}+\|b(t_*)\|_{\dot H^{\beta-1}}>\|a(t_*)\|_{\dot H^\beta}\ >\ \Lambda,
\]
while the norm of $b$ remains small:
\[
\|b(t)\|_{H^{\beta-1}}<\epsilon, \quad \forall \ t \in [0,t_*].
\]
\end{theorem}
\begin{remark}
The result is named ``half" strong ill-posedness because of the smallness of $\|b\|_{H^{\beta-1}}$. In fact, if every dissipation term is dropped in our proof of the theorem below, Theorem \ref{thm:main} extends to inviscid $2\frac{1}{2}$D electron MHD, giving half strong ill-posedness when $3<\beta<4$.
\end{remark}
\begin{remark}\label{remark-comparison-jeong-oh}
We may emphasize that our result is different to the result by Jeong and Oh \cite{Jeong-Oh-2.5}, where ill posedness of \textit{3D} electron MHD is established in $H^s$ for $s\ge 3$ with $0\le \alpha < 1/2$. Recall under the $2\frac{1}{2}$D symmetry (i.e. $z$-independence of $B$), we can write $B = (\nabla \times ae_z)+be_z$, and Theorem \ref{thm:main} suggests ill posedness for $B \sim (\nabla a,b)$ in $H^s$ for $2<s<3-2\alpha$. Since the $2\frac{1}{2}$D symmetry is preserved under evolution of electron MHD, we may interpret Theorem \ref{thm:main} as an ill posedness result in $\mathbb R^2_{(x,y)} \times \mathbb T_z$ by setting $B(t,x,y,z) \equiv B(t,x,y)$.
\end{remark}
\begin{remark}
Although our theorem may be viewed as a small improvement of result by Jeong and Oh \cite{Jeong-Oh-2.5} as mentioned in Remark \ref{remark-comparison-jeong-oh}, we regret to point out that our result has yet to overcome the barrier of $\alpha = 1/2$, which may need more delicate estimate of error.
\end{remark}

Also, let us clarify the notion of strong ill-posedness in this paper. Consider the ill-posedness problem in space $X$ of a PDE with solution $u$, we say that the PDE is strongly ill-posed under the first definition below:
\begin{itemize}
    \item (Strong ill-posedness) For any small $\epsilon>0$ and $t_*>0$, and any large $\Lambda>0$, we can find initial data $u_0$ such that $\|u_0\|_X < \epsilon$, and its solution $u$ inflates in a short time: $\|u(\cdot,t_*)\|_X > \Lambda$.
    \item (Non-existence) For any small $\epsilon>0$ and $t_*>0$, we can find initial data $u_0$ such that $\|u_0\|_X < \epsilon$ whose solution is not in the space $X$: $\|u(\cdot,t)\|_X = \infty \ \ \forall t \in(0,t_*]$ .
\end{itemize}
In this paper, we proved the strong ill-posedness as defined above. Non-existence is even \textit{stronger} than \textit{strong} ill-posedness, but our approach stops at strong ill-posedness without extending to non-existence. This is because typically one gets non-existence from strong ill-posedness by a gluing argument (see, e.g. \cite{bourgainLi15} for Euler and \cite{córdoba2025strongillposednessnonexistencesobolev} for gSQG). However, our initial data is a single profile without gluing. It is very possible that one could use the translational invariance of electron MHD to gain non-existence by gluing a sequence of our initial data apart from each other, but we are not going to pursue it here.

We may outline the ideas in the proof of Theorem \ref{thm:main}. It is well known that if one has a non-Lipchitz drift term (in our case, it is $u := \nabla^\perp b$), then regularity tends to be lost (see, e.g. the classical works \cite{Diperna1989OrdinaryDE} and \cite{Diperna1987OscillationsAC}). So, we will construct small initial data with $u_0=\nabla^\perp b_0$ not Lipschitz. After that, we will construct approximation solutions $\bar a$ and $\bar u = \nabla^\perp \bar b$ and show the norm inflation of $\bar a$ and the smallness of $\bar u$. Lastly, we will show the difference between the real solution and the approximation $\bar a-a$ is controlled and that $u$ is small.

\medskip
\noindent\textbf{Acknowledgments.}   H.Z. is partially supported by the National Science Foundation grant DMS-2308208. The authors are also grateful to discussions with Mimi Dai.

\section{Preliminaries} \label{sec-prelim}
\noindent\textbf{Notations.}
We may write $f\lesssim g$ if $f < C g$ up to a constant $C$. For $f \gtrsim g$ it is similar. If $g \lesssim f \lesssim g$, then we will write $f \simeq g$. We will also write $\lambda, N \gg 1$ if parameters $\lambda, N$ are arbitrarily large.

\noindent\textbf{Sobolev space and polar gradient.}
We may recall the following definition of Sobolev space and norm
\[
\|f\|_{W^{s,p}} = \|J^s f\|_{L^p}
\]
where $J$ is differential operator with $\mathcal F (Jf)(\xi)= \sqrt{1+|\xi|^2}\mathcal F (f)(\xi)$ and $\mathcal F (f)$ is the Fourier transform of $f$. If we replace $J$ by $\Lambda$ such that $\mathcal F(\Lambda f)= |\xi|\mathcal F (f)(\xi)$, then we have the homogeneous Sobolev norm
\[
\|f\|_{\dot W^{s,p}} = \|\Lambda^s f\|_{L^p}.
\]
In particular, when $p=2$, we write $H^s := W^{s,2}$ and $\dot H^s := \dot W^{s,2}$. Also, on $\mathbb R^2$ we may recall the gradient in polar coordinate
\[
\nabla f = \partial_r f e_r + \frac{\partial_\theta f}{r} e_\theta.
\]

\noindent\textbf{Basic energy estimate.}
By a standard estimate of $2\frac{1}{2}$D electron MHD (\ref{eq:resistive-25d}) using divergence free of drift velocity $\nabla^\perp b$, we have $a \in L^\infty_tL^2_x$. For higher norm, we also has a energy estimate:
\begin{equation}\notag
\begin{split}
&\frac12\frac{d}{dt}\int_{\mathbb R^2}|\nabla a|^2+|b|^2\,dxdy+\int_{\mathbb R^2}\mu|\Lambda^{\alpha+1} a|^2+\nu|\Lambda^{\beta} b|^2\,dxdy\\
=& \int_{\mathbb R^2}\nabla^\perp b\cdot\nabla a\Delta a\, dxdy+ \int_{\mathbb R^2}(\nabla^\perp \Delta a\cdot\nabla a) b\, dxdy\\
=&-  \int_{\mathbb R^2} b\nabla a\cdot \nabla^\perp\Delta a\, dxdy+ \int_{\mathbb R^2}(\nabla^\perp \Delta a\cdot\nabla a) b\, dxdy\\
=&\ 0.
\end{split}
\end{equation}
Hence we have a priori estimate $(a, b)\in L^\infty_t(H^1_x \times L^2_x)$. Therefore, we should study norm inflation in $H^\beta\times H^{\beta-1}$ when $\beta>1$.

\section{Initial data, approximate solution, and approximate norm inflation}
From now on, we will denote $u :=\nabla^\perp b$ as the drift velocity, and then study the following,
\begin{equation}\label{eq:resistive-25d-u}
\begin{cases}
a_t + u\cdot \nabla a + \mu(-\Delta)^{\alpha} a = 0,\\[2pt]
u_t + \nabla^\perp(\nabla^{\perp}a\cdot\nabla \Delta a) + \nu(-\Delta)^{\alpha} u = 0.
\end{cases}
\end{equation}

\medskip
\noindent\textbf{Choice of initial data.}
Assume $\beta < 4-2\alpha$ and let $\lambda\gg1$ and $N \gg 1$ be large parameters. In fact, one may relate $\lambda$ to $N$ using powers
\[
\lambda^{\beta+2\alpha-4}=N^{-100}.
\]
Pick smooth radial cutoffs $g \in C_c^\infty ((2,3))$ and $h\in C_c^\infty((1,4))$ with $h',h'' > 0$ on $[2,3]$. In polar coordinates $(r,\theta)$ set
\begin{equation}\label{eq:ID}
\begin{split}
&a_{0}(r,\theta)= \epsilon\lambda^{1-\beta}N^{-\beta}g(\lambda r)\cos\left( N \theta \right), \qquad b_{0}(r,\theta)=\epsilon\lambda^{2-\beta}h(\lambda r), \\
&u_0:=\nabla^\perp b_0=\partial_r b_0\,e_\theta=\epsilon\lambda^{3-\beta}h'(\lambda r)e_{\theta}.
\end{split}
\end{equation}
\begin{lemma}[Initial data bounds]\label{est-u0}
Let $s\ge 0$ and let $(a_0,b_0,u_0)$ be defined by \eqref{eq:ID}, then, uniformly in $\lambda,N\gg 1$,
\begin{equation}\label{eq:est-u0}
\|a_0\|_{H^s}\ \lesssim\ \epsilon\lambda^{s-\beta}N^{s-\beta},
\qquad
\|u_0\|_{H^s}\ \lesssim\ \epsilon\lambda^{s+2-\beta},
\qquad
\|u_0\|_{C^1}\ \simeq\ \epsilon\lambda^{4-\beta}.
\end{equation}
\end{lemma}

\begin{proof}
Note the support of $g(\lambda r)$ and $h(\lambda r)$ is on the annulus $r \sim \lambda^{-1}$. So $|\text{supp }a_0| = |\text{supp }u_0| \simeq \lambda^{-2}$. It can be easily check that
\[
\|a_0\|_{L^\infty} \simeq \epsilon \lambda^{1-\beta} N^{-\beta}, \qquad \|u_0\|_{L^\infty} \simeq \epsilon \lambda^{3-\beta}.
\]
With $\nabla = \partial_r + \frac{\partial_\theta}{r} \sim \partial_r + \lambda\partial_\theta$ due to $r \sim \lambda^{-1}$, each derivative on $a_0$ gains a factor of $\lambda N$ and each derivative on $u_0$ gains a factor of $\lambda$. To be more precise,
\[
\|\nabla a_0\|_{L^\infty} \lesssim \epsilon \lambda^{2-\beta} N^{1-\beta}, \qquad \|\nabla u_0\|_{L^\infty} \simeq \epsilon \lambda^{4-\beta},
\]
where we use $\lesssim$ instead of $\simeq$ in the inequality of $a_0$ since we do not exclude higher derivatives of $g$ being zero. Adding two estimates of $u_0$, we proved $\|u_0\|_{C^1}\simeq\epsilon\lambda^{4-\beta}$. More generally,
\[
\|D^m a_0\|_{L^\infty} \lesssim \epsilon \lambda^{1+m-\beta} N^{m-\beta}, \qquad \|D^m u_0\|_{L^\infty} \lesssim \epsilon \lambda^{3+m-\beta}.
\]
Then using $|\text{supp}| \sim \lambda^{-2}$ and $L^2$-$L^\infty$ inequality yields
\begin{align*}
\|D^m a_0\|_{L^2} \ &\le (\lambda^{-2})^{1/2}\|D^m a_0\|_{L^\infty} \lesssim \epsilon \lambda^{m-\beta}N^{m-\beta} \\
\|D^m u_0\|_{L^2} \ &\le (\lambda^{-2})^{1/2}\|D^m u_0\|_{L^\infty} \lesssim \epsilon \lambda^{m-\beta}.
\end{align*}
Add from $0$ to $m$, we have
\[
\|a_0\|_{H^m} \lesssim \epsilon \lambda^{m-\beta}N^{m-\beta}, \qquad \|u_0\|_{H^m} \lesssim \epsilon \lambda^{m-\beta}.
\]
For non-integer $s$, interpolating the above estimates concludes the proof.
\end{proof}
Notice that as $\lambda \to \infty$, the $C^1$ norm of $u_0$ will go to infinity thus will be non-Lipchitz, leading to regularity loss and our norm inflation below.

\medskip
\noindent\textbf{Approximate solution.}
We will approximate the real solution in a short time by inviscid transport. Let $\bar a$ and $\bar u$ be the solution of
\begin{equation}\label{eq:approx}
\begin{split}
&\bar a_t+u_0\cdot\nabla \bar a = 0,\qquad \bar a(0)=a_0, \\
&\bar u_t+\nabla^\perp(\nabla^\perp \bar a\cdot\nabla \Delta \bar a) = 0,\qquad \bar u(0)=u_0,
\end{split}
\end{equation}
so explicitly in polar coordinates $(r,\theta)$
\begin{equation}\label{eq:tilde-a-formula}
\bar a(t,r,\theta)=a_0\!\left(r,\ \theta-\frac{\partial_r b_0}{r}\,t\right)
= \epsilon\lambda^{1-\beta}N^{-\beta}\,
g(\lambda r)\,
\cos\Big(N\theta-\,\epsilon N\lambda^{4-\beta} t h'(\lambda r)\Big).
\end{equation}
The fast angular shear $\frac{\partial_{r}b_{0}}{r} = \epsilon \lambda^{3-\beta} \frac{h'(\lambda r)}{ r} \sim \epsilon \lambda^{4-\beta} h'(\lambda r)$ on the annulus $r \sim \lambda^{-1}$ drives norm inflation for $\bar a$ at time $t_*=\epsilon^{-2}\lambda^{\beta-4}$.
To be more precise, we have the following:
\begin{lemma}[Approximate solution inflation]\label{le-a-osc}
Set $t_* := \epsilon^{-2}\lambda^{-4+\beta}$ which goes to zero with $\lambda$ growing. For $s\ge0$, $1\le p \le \infty$, and $0\le t \le t_*$, the function $\bar a$ given by \eqref{eq:tilde-a-formula} satisfies
\begin{equation}\label{eq-abar-Wsp}
\lVert \bar{a}(t) \rVert_{\dot{W}^{s,p}} \lesssim \epsilon^{1-s} (\lambda N)^s \lambda^{1-\frac{2}{p}-\beta}N^{-\beta},
\end{equation}
and when $t=t_*$
\begin{equation}\label{eq-abar-diego}
\|\bar a(t_*)\|_{\dot{H}^s}\simeq \epsilon^{1-s}(\lambda N)^s\lambda^{-\beta}N^{-\beta}=\epsilon^{1-s}\lambda^{s-\beta}N^{s-\beta}.
\end{equation}
In particular, when $s=\beta$, for any $\Lambda>0$ there exists $\epsilon$ small such that $\|\bar a (t_*)\|_{\dot H^\beta}\simeq \epsilon^{1-\beta} > \Lambda$.
\end{lemma}

\begin{proof}
When $s=0$, it is trivial that
\[
\|\bar a\|_{L^\infty} \simeq \epsilon\lambda^{1-\beta}N^{-\beta}.
\]
Then, each derivative gains a factor of $|\epsilon N\lambda^{4-\beta}t_* \lambda h''(\lambda r)|\simeq_h\epsilon \lambda^{5-\beta} Nt_*=\epsilon^{-1}\lambda N$,
\[
\lVert \bar{a}(t) \rVert_{\dot{W}^{s,\infty}} \lesssim \epsilon^{1-s} (\lambda N)^s \lambda^{1-\beta}N^{-\beta},
\]
which proves \eqref{eq-abar-Wsp} when $p=\infty$. Note the $|\text{supp} \ \bar a| \simeq \lambda^{-2}$, thus
\[
\lVert \bar{a}(t) \rVert_{\dot{W}^{s,p}} \le |\text{supp}\ \bar a|^{1/p}\lVert \bar{a}(t) \rVert_{\dot{W}^{s,\infty}} \lesssim \epsilon^{1-s} (\lambda N)^s \lambda^{1-\frac{2}{p}-\beta}N^{-\beta}.
\]
So \eqref{eq-abar-Wsp} is proved. To prove \eqref{eq-abar-diego}, we utilize \cite[Lemma 8]{córdoba2023}. Take variable $\rho = \lambda r$, the lemma implies that at $t_*=\epsilon^{-2}\lambda^{\beta-4}$,
\[
\|\epsilon \lambda^{1-\beta}N^{-\beta}g(\rho)\cos(N\theta-\epsilon^{-1} Nh'(\rho))\|_{\dot H^{-s}} \lesssim (\epsilon^{-1} N)^{-s}\|\epsilon\lambda^{1-\beta}N^{-\beta}g(\rho)\|_{C^1} 
\]
\[
= \epsilon^{1+s} N^{-s}\lambda^{1-\beta}N^{-\beta}.
\]
Thus
\[
\|\bar a(t_*,r,\theta)\|_{\dot H^{-s}} =\lambda^{-s-1}\|\epsilon \lambda^{1-\beta}N^{-\beta}g(\rho)\cos(N\theta-\epsilon^{-1} Nh'(\rho))\|_{\dot H^{-s}}\lesssim \epsilon^{1+s}\lambda^{-s-\beta}N^{-s-\beta}.
\]
Note that
\begin{align*}
\|\bar a(t_*)\|_{L^2}^2 =& \epsilon^2\lambda^{2-2\beta}N^{-\beta} \int_0^{2\pi}\int_0^\infty g^2(\lambda r)\cos^2 (N\theta-\epsilon Nh'(\lambda r)) rdrd\theta \\
=& \epsilon^2\lambda^{-2\beta}N^{-\beta} \int_0^{2\pi}\int_0^\infty g^2(\rho)\cos^2 (N\theta-\epsilon Nh'(\rho)) \rho d\rho d\theta \\
\simeq& \epsilon^2\lambda^{-2\beta}N^{-\beta}
\end{align*}
which yields $\|\bar a(t_*)\|_{L^2} \simeq \epsilon \lambda^{-\beta}N^{-\beta}$. Then we have
\[
\|\bar a(t_*)\|_{\dot H^s} \gtrsim \|\bar a(t_*)\|_{\dot H^{-s}}^{-1} \|\bar a(t_*)\|_{L^2}^2 \gtrsim \epsilon^{1-s}\lambda^{s-\beta}N^{s-\beta},
\]
which completes the proof.
\end{proof}
At the moment $\epsilon$ is fixed to give small initial data and norm inflation. In the following, we will not track the dependence of $\epsilon$ and will deal with only large parameters $\lambda$ and $N$ unless otherwise stated.

\begin{lemma}[Smallness of $\bar u$]\label{small-ubar}
    For $0 \le t \le t_*$ and $\beta>2$,
    \[
    \|\bar u\|_{H^{\beta-2}} \lesssim \|u_0\|_{H^{\beta-2}}.
    \]
\end{lemma}
\begin{proof}
Notice that
\[
\bar u(t)=u_0+\int_0^t \nabla^{\perp}(\nabla^{\perp}\bar a(\tau)\cdot\nabla\Delta\bar a(\tau)) \, d\tau.
\]
Thus,
\begin{equation}\notag
\begin{split}
\|\bar u(t)\|_{L^2}\leq& \ \|u_0\|_{L^2}+t\|\nabla^{\perp}(\nabla^{\perp}\bar a\cdot\nabla\Delta\bar a)\|_{L^2}
\end{split}
\end{equation}
For the last term, set 
\[
\bar a (t,r,\theta) = \lambda^{1-\beta}g(\lambda r)\cos(Nf(t,\lambda r,\theta))
\]
where $f(t,\lambda r,\theta)=\theta - \lambda^{4-\beta}th'(\lambda r)$. Then for a differential operator $D$ such as $\nabla$,
\[
D \bar a(t_*) = \lambda^{2-\beta}g'(\lambda r) (D r)\cos(Nf(t_*,\lambda r,\theta))-\lambda^{1-\beta}g(\lambda r) \sin (Nf(t_*,\lambda r,\theta)) N\lambda  D f(t_*,\lambda r,\theta).
\]
Note the dominating terms, i.e. the terms with most gaining factors, are the last term, which comes from differentiation on trigonometric function with $f(t,\lambda r,\theta)$ inside. So, $\nabla^\perp \bar a$ and $\nabla \Delta \bar a$ have dominating terms as the following
\begin{align*}
\nabla^\perp \bar a(t_*,\lambda r,\theta) &= -\lambda^{1-\beta}g(\lambda r) \sin (Nf(t_*,\lambda r,\theta)) \lambda N \nabla^\perp f(t_*,\lambda r,\theta) + o(\lambda N), \\
\nabla\Delta \bar a(t_*,\lambda r,\theta) &= \lambda^{1-\beta}g(\lambda r) \sin (Nf(t_*,\lambda r,\theta)) (\lambda N)^3 |\nabla f(t_*,\lambda r,\theta)|^2 \nabla f(t_*,\lambda r,\theta) + o((\lambda N)^3).
\end{align*}
Then note the dominating terms of $\nabla^\perp \bar a$ and $\nabla \Delta \bar a$ are orthogonal
\begin{equation}\label{eq-cancel-orthog}
\lambda^{2-2\beta}g^2(\lambda r) \sin^2 (Nf(t_*,\lambda r,\theta)) (N\lambda)^4 |\nabla f(t_*,\lambda r,\theta)|^2 \nabla^\perp f(t_*,\lambda r,\theta)\cdot\nabla f(t_*,\lambda r,\theta) = 0,
\end{equation}
which allows us to lower the power of dominating gaining factors. Then, our dominating gaining factor of $\nabla^\perp \bar a \cdot \nabla \Delta \bar a$ becomes $\lambda^4 N^3$ instead of $(\lambda N)^4$, losing a factor of $N$. The power of $\lambda$ is not lowered due to the high derivative on $g(\lambda r)$. Note $\|\bar a\|_{L^\infty}\lesssim \lambda^{1-\beta}N^{-\beta}$ with $|\text{supp} \ \bar a| \sim \lambda^{-2}$, so for $\nabla^{\perp}(\nabla^{\perp}\bar a\cdot\nabla\Delta\bar a)$ on $0\le t\le t_*$,
\[
\|\nabla^{\perp}(\nabla^{\perp}\bar a\cdot\nabla\Delta\bar a)\|_{L^2}
\lesssim (\lambda^{-2})^{1/2}\|\nabla^{\perp}(\nabla^{\perp}\bar a\cdot\nabla\Delta\bar a)\|_{L^\infty} \lesssim \lambda^5 N^4\lambda^{1-2\beta}N^{-2\beta} = \lambda^{6-2\beta}N^{4-2\beta}.
\]
Thus
\begin{equation}\label{bar-u-tN}
\begin{split}
\|\bar u(t)\|_{L^2}\le \|u_0\|_{L^2} + t_*\|\nabla^{\perp}(\nabla^{\perp}\bar a\cdot\nabla\Delta\bar a)\|_{L^2} \lesssim \|u_0\|_{L^2} + \lambda^{\beta-4} \lambda^{6-2\beta}N^{4-2\beta}.
\end{split}
\end{equation}
For higher norm than \eqref{bar-u-tN}, add a gaining factor of $\lambda N$ for each derivative,
\begin{equation}\label{bar-u-tN-2}
\begin{split}
\|\bar u(t)\|_{H^{\beta-2}}\lesssim \|u_0\|_{H^{\beta-2}}
+\lambda^{\beta-4} (\lambda N)^{\beta-2} \lambda^{6-2\beta}N^{4-2\beta} = \|u_0\|_{H^{\beta-2}} + N^{2-\beta}
\end{split}
\end{equation}
whose last term vanishes as $N$ grows if $\beta > 2$.
\end{proof}

\begin{remark}
In order to utilize the above orthogonal cancellation, we need the gaining factor $\lambda N$ from the trigonometric function that dominates $\lambda$, which is the reason why one cannot simply set $N=1$.
\end{remark}
\begin{remark}
This lemma is not necessary since we will directly prove the smallness of $u$. However, we keep it as an illustration of earning a small factor through orthogonal cancellation, which is the key to controlling error in Lemma \ref{le-perturbation} below.
\end{remark}

\section{Error equation and estimate of error}

Recall $u=\nabla^\perp b$, and set
\begin{equation}\label{eq:def-A}
A:=a-\bar a.
\end{equation}
Subtracting \eqref{eq:approx} from the $a$-equation in \eqref{eq:resistive-25d-u} gives
\begin{equation}\label{eq:A-eqn}
A_t+u\cdot\nabla A+\mu(-\Delta)^\alpha A
=-(u-u_0)\cdot\nabla\bar a\;-\;\mu(-\Delta)^\alpha\bar a,\qquad A(0)=0.
\end{equation}
We proceed by deriving the energy estimate of $A$ as the following:
\begin{lemma}[Energy for $A$]\label{lem:A-energy}
Let $A$ solve \eqref{eq:A-eqn}. Then for $m \ge 0$ and $1\le p \le \infty$,
\begin{align}
\frac{d}{dt}\|A\|_{\dot{H}^{m}} \lesssim& \sum_{m'=1}^m\| u\|_{\dot{W}^{m',\infty}}\|A\|_{\dot{H}^{m-m'+1}}
+ 
\sum_{m'=0}^m\|u-u_0\|_{\dot{H}^{m'}}\|\bar a\|_{\dot{W}^{m-m'+1,\infty}} + \mu\|\Lambda^{2\alpha+m} \bar a\|_{L^2}.
\end{align}
Using Gr\"onwall, for $m=0,1$,
\[
\|A\|_{\dot{H}^m} \lesssim \int_0^t  \left(\sum_{m'=0}^m\|u(\tau)-u_0\|_{\dot{H}^{m'}}\|\bar{a}(\tau)\|_{\dot{W}^{m-m'+1,\infty}} + \mu\|\Lambda^{2\alpha+m} \bar a(\tau)\|_{L^2} \right) e^{\int_\tau^t\|u(\tau')\|_{\dot{W}^{1,\infty}}d\tau'}d\tau,
\]
and for $m \ge 2$,
\begin{align*}
\|A\|_{\dot{H}^m} \lesssim \int_0^t&  \big(\sum_{m'=2}^m\|u\|_{\dot{W}^{m',\infty}}\|A\|_{\dot{H}^{m-m'+1}}+\sum_{m'=0}^m\|u(\tau)-u_0\|_{\dot{H}^{m'}}\|\bar{a}(\tau)\|_{\dot{W}^{m-m'+1,\infty}} \\
&+ \mu\|\Lambda^{2\alpha+m} \bar a(\tau)\|_{L^2} \big)e^{\int_\tau^t\|u(\tau')\|_{\dot{W}^{1,\infty}}d\tau'}d\tau.
\end{align*}
\end{lemma}

\begin{proof}
Notice that
\begin{align*}
\frac{1}{2}\frac{d}{dt}\|D^mA\|_{L^2}^2+\mu\|\Lambda^{\alpha+m} A\|_{L^2}^2
=&-\int_{\mathbb R^2} D^mA D^m(u\cdot\nabla A)-\int_{\mathbb R^2} D^m A D^m\left((u-u_0)\cdot\nabla\bar a\right) \\
&-\mu \int_{\mathbb R^2} D^m A D^m(-\Delta)^\alpha\bar a.    
\end{align*}
Now study right side term by term:
\begin{align*}
-\int_{\mathbb R^2} D^mA D^m(u\cdot\nabla A) =& \sum_{m'=1}^m \int_{\mathbb R^2} D^m A \left( D^{m'}u \cdot D^{m-m'} \nabla A \right) \\
\le& \|D^mA\|_{L^2} \sum_{m'=1}^m\|D^{m'} u\|_{L^\infty}\|\nabla A\|_{H^{m-m'}}
\end{align*}
where if $m'=0$, then divergence-free property of $u$ will lead to cancellation. And similarly
\begin{align*}
-\int_{\mathbb R^2} D^m A D^m\left((u-u_0)\cdot\nabla\bar a\right) &\le \|D^mA\|_{L^2}\sum_{m'=0}^m\|u-u_0\|_{H^{m'}}\|\nabla D^{m-m'} \bar a\|_{L^\infty} \\
-\mu \int_{\mathbb R^2} D^m A D^m(-\Delta)^\alpha\bar a &\le \mu \|D^m A\|_{L^2}\|\Lambda^{2\alpha+m} \bar a\|_{L^2}.
\end{align*}
Collecting all of these, droping the dissipation on the left and dividing both sides by $\|D^m A\|_{L^2}$ completes the proof.
\end{proof}

A key point is that the estimate for $A$ in Lemma \ref{lem:A-energy} requires control of both $u-u_0$ and $u$. In turn, these quantities are driven by $a$ through the $u$-equation. More precisely, for any $m\ge 0$, the standard $\dot H^m$ energy estimate gives
\begin{equation}\label{eq-de-u}
\begin{split}
\frac12\frac{d}{dt}\|u\|_{\dot H^m}^2+\nu\|\Lambda^\alpha u\|_{\dot H^m}^2
&\le
\bigl\|\nabla^\perp(\nabla^\perp a\cdot \nabla\Delta a)\bigr\|_{\dot H^m}\,\|u\|_{\dot H^m}.
\end{split}
\end{equation}
Dropping the dissipative term and integrating in time, we obtain formally, and then by continuity,
\begin{equation}\label{eq:u-growth-heuristic}
\begin{split}
\|u(t)-u_0\|_{\dot H^m}
&\le \int_0^t \bigl\|\nabla^\perp(\nabla^\perp a(\tau)\cdot \nabla\Delta a(\tau))\bigr\|_{\dot H^m}\,d\tau,\\
\|u(t)\|_{\dot H^m}
&\le \|u_0\|_{\dot H^m}
+\int_0^t \bigl\|\nabla^\perp(\nabla^\perp a(\tau)\cdot \nabla\Delta a(\tau))\bigr\|_{\dot H^m}\,d\tau.
\end{split}
\end{equation}
We may lower the power of gaining factors of the main integral part like in \eqref{eq-cancel-orthog}. When $m=0$,
\begin{align*}
\int_0^{t_*}\lVert \nabla^\perp (\nabla^\perp a \cdot \Delta \nabla a)\rVert_{L^2}  dt
\leq& t_* \lVert \nabla^\perp (\nabla^\perp \bar{a} \cdot \Delta \nabla \bar{a})\rVert_{L^2}  \\
&+t_*\lVert\nabla^\perp (\nabla^\perp (a-\bar{a}) \cdot \Delta \nabla \bar{a}) \rVert_{L^2}  \\
&+t_*\lVert\nabla^\perp (\nabla^\perp a \cdot \Delta \nabla (a-\bar{a}))\rVert_{L^2}.
\end{align*}
For larger $m>0$, above is multiplied by the gaining factor $(\lambda N)^m$. The first term is controlled as \eqref{eq-cancel-orthog}. For the other two, in turn we need $A=\bar a - a$ to be controlled. Thus, to break the loop, we will use a bootstrap argument assuming $A$ is initially small in the proof below.

\begin{lemma}\label{le-perturbation}
Let $3 < \beta<4-2\alpha$, $0\le\alpha<1/2$. There exist large $\lambda>0$ and $N>0$ such that for $s \ge 0$,
\begin{equation}\label{eq-est-A}
\|A(t_*)\|_{\dot{H}^s} \lesssim \lambda^{s-\beta}N^{s-\beta-1},
\end{equation}
and
\begin{equation}\label{eq-est-u}
\|u(t)\|_{H^{\beta-2}}\lesssim \|u_0\|_{H^{\beta-2}},
\qquad
\forall\, t\in[0,t_*].
\end{equation}
\end{lemma}
\begin{proof}
The main strategy is a bootstrap argument. Since $A(0)=0$, we can safely assume in a short time $0\le t\le t_0$ for some $M>0$ we have the following
\begin{equation}\label{bootsrap}
\boxed{\textbf{Bootstrap assumption:} \quad \|A\|_{\dot H^{s}} \le 2M \lambda^{s-\beta}N^{s-\beta-1}, \quad t \in[0,t_{0}]. }
\end{equation}
Our goal is to show 
\[
\|A\|_{\dot H^{s}} \le \epsilon^{100-2s}M \lambda^{s-\beta}N^{s-\beta-1}, \quad t \in[0,t_{0}],
\]
which would close the bootstrap, and we could extend $t_0$ to $t_*$ through a continuity argument. In particular, then \eqref{eq-est-A} is proved. The choice of a smaller constant $\epsilon^{100-2s}M < 2M$ is slightly strange, but is needed later in induction. Since $\epsilon$ is already fixed while $\lambda$ and $N$ are free to be large, below we will not track $\epsilon$ until the estimate of $\|A\|_{\dot H^m}$ for $m \ge 2$, where the strange bootstrap constant is used. With \eqref{bootsrap}, and Lemma \ref{le-a-osc}, we have
\begin{equation}\label{eq-bootstrap-a}
    \|a\|_{\dot H^{s}} \le \|\bar a\|_{\dot H^{s}} + \|A\|_{\dot H^{s}} \lesssim \lambda^{s-\beta}N^{s-\beta}, \quad t \in[0,t_{0}].
\end{equation}
With the bootstrap assumption above, we could first prove \eqref{eq-est-u} like in Lemma \ref{small-ubar}, while now we cannot use the orthogonal cancellation brought by the structure of $\bar a$. From \eqref{eq-de-u},
\begin{align*}
\|u(t)\|_{L^2} \le& \|u_0\|_{L^2} + t_*\|\nabla^{\perp}(\nabla^{\perp} a\cdot\nabla\Delta a)\|_{L^2} \\
\le& \|u_0\|_{L^2} + t_*\sum_{k=1,2} \|D^k a\|_{L^{\infty}} \|D^{5-k}a\|_{L^2}  \\
\lesssim& \|u_0\|_{L^2} + \lambda^{\beta-4} \lambda^{2-\beta}N^{1-\beta}\lambda^{4-\beta}N^{4-\beta} \\
\le& \|u_0\|_{L^2} + \lambda^{2-\beta}N^{5-2\beta}
\end{align*}
where whether $k=1$ or $k=2$ makes no difference since the total gaining factors are the same. For higher norms, multiplying gaining factors $(\lambda N)^{\beta-2}$ yields
\[
\|u(t)\|_{H^{\beta-2}} \lesssim \|u_0\|_{H^{\beta-2}} + N^{3-\beta}, \quad t \in[0,t_{0}],
\]
whose last term vanishes as $N$ grows if $\beta>3$. Thus \eqref{eq-est-u} is proved if we could close the bootstrap.

\noindent$\mathbf{L^2}$ \textbf{case:}
\[
\|A\|_{L^2} \lesssim \int_0^t  \left(\|u(\tau)-u_0\|_{L^2}\|\bar{a}(\tau)\|_{\dot{W}^{1,\infty}} + \mu\|\Lambda^{2\alpha+1} \bar a(\tau)\|_{L^2} \right) e^{\int_\tau^t\|u(\tau')\|_{\dot{W}^{1,\infty}}d\tau'}d\tau.
\]
For the exponential, Sobolev embedding yields
\begin{equation}
\begin{split}
\int_0^{t_*} \|u\|_{\dot W^{1,\infty}} \le& t_* \|u\|_{\dot W^{1,\infty}} \\
\lesssim& t_*\int_0^{t_*}\|\nabla^\perp(\nabla^\perp a \cdot \nabla \Delta a)\|_{ H^{2+\delta}} + t_* \|u_0\|_{\dot W^{1,\infty}} \\
\lesssim& t_*^2 \|\nabla^\perp(\nabla^\perp \bar a \cdot \nabla \Delta \bar a)\|_{ H^{2+\delta}} + t_*^2 \|\nabla^\perp(\nabla^\perp A \cdot \nabla \Delta \bar a)\|_{H^{2+\delta}} \\
&+ t_*^2 \|\nabla^\perp(\nabla^\perp a \cdot \nabla \Delta A)\|_{H^{2+\delta}} + t_* \|u_0\|_{\dot W^{1,\infty}}
\end{split}
\end{equation}
for any small $\delta>0$. The first term can be estimated as in \eqref{eq-cancel-orthog}, lowering one power of $N$,
\[
t_*^2 \|\nabla^\perp(\nabla^\perp \bar a \cdot \nabla \Delta \bar a)\|_{H^{2+\delta}} \lesssim \lambda^{2\beta-8} \lambda^{6-2\beta}N^{5-1-2\beta}(\lambda N)^{2+\delta} = \lambda^{\delta}N^{6+\delta-2\beta},
\]
and the two terms in the middle are similar: Although now we cannot lower a power of $N$ using orthogonality, we can still lower a power using bootstrap assumption \eqref{bootsrap}. Together with estimate of $\bar a$ \eqref{le-a-osc}, and estimate of $a$ \eqref{eq-bootstrap-a}, we have
\begin{align*}
&t_*^2 \|\nabla^\perp(\nabla^\perp A \cdot \nabla \Delta \bar a)\|_{H^{2+\delta}} + t_*^2 \|\nabla^\perp(\nabla^\perp a \cdot \nabla \Delta A)\|_{H^{2+\delta}}  \\
\lesssim& t_*^2\|D A\|_{L^2}\|D^{6+\delta} \bar a\|_{L^\infty} + t_*^2 \|D^{3}A\|_{L^2}\|D^{4+\delta}a\|_{H^{1+\delta}}  \\
\lesssim& \lambda^{2\beta-8} (\lambda N)\lambda^{-\beta}N^{-\beta-1} (\lambda N)^{6+\delta} \lambda^{1-\beta}N^{-\beta} + \lambda^{2\beta-8} (\lambda N)^{3}\lambda^{-\beta}N^{-\beta-1} (\lambda N)^{5+2\delta} \lambda^{-\beta}N^{-\beta} \\
\lesssim& \lambda^{2\beta-8} \lambda^{6-2\beta}N^{5-1-2\beta} (\lambda N)^{2+2\delta}=\lambda^{2\delta}N^{6+2\delta-2\beta}.
\end{align*}
Note $A$, $a$, $\bar a$ have the same gaining factor $\lambda N$, so the way to distribute derivatives does not matter. Moreover, the last term is harmless since $t_* \gtrsim \frac{1}{\|u_0\|_{C^1}}$, and
\[
t_* \|u_0\|_{\dot W^{1,\infty}} \lesssim \lambda^{\beta-4}\lambda^{4-\beta}=1.
\]
If we choose $\delta=\delta(\lambda) = \min \{ \log_\lambda 2,\log_N 2\}$ small, then
\begin{equation}\label{eq-exp-bdd}
\int_0^{t_*} \|u\|_{\dot W^{1,\infty}} \lesssim t_*\|u\|_{\dot W^{1,\infty}} \lesssim N^{6-2\beta} + 1 \lesssim 1
\end{equation}
if $\beta\ge3$. Thus,
\[
e^{\int_0^{t_*} \|u\|_{\dot W^{1,\infty}}} \lesssim 1.
\]
Also note that
\begin{equation}\label{eq-est-u-u_0-L2}
\begin{split}
\|u-u_0\|_{L^2} \lesssim& t_* \|\nabla^\perp(\nabla^\perp \bar a \cdot \nabla \Delta \bar a)\|_{L^2} + t_* \|\nabla^\perp(\nabla^\perp A \cdot \nabla \Delta \bar a)\|_{L^2}
+ t_* \|\nabla^\perp(\nabla^\perp a \cdot \nabla \Delta A)\|_{L^2} \\
\lesssim& \lambda^{\beta-4}\lambda^{5-2\beta}N^{4-2\beta} + t_* \|D A\|_{L^2} \|D^4 \bar a\|_{L^\infty} + t_* \|D^3 A\|_{L^2} \|D^2 a\|_{H^{1+\delta}} \\
\lesssim& \lambda^{\beta-4}\lambda^{6-2\beta}N^{4-2\beta} + \lambda^{\beta-4}\lambda^{6-2\beta+\delta}N^{4-2\beta+\delta} \\\lesssim& \lambda^{2-\beta+\delta}N^{4-2\beta+\delta} \simeq \lambda^{2-\beta}N^{4-2\beta}
\end{split}
\end{equation}
under the $\delta$ chosen above. Again, $A$, $a$, $\bar a$ have the same gaining factor $\lambda N$, so the way to distribute derivatives does not matter. Finally, the $L^2$ norm satisfies
\begin{align*}
\|A\|_{L^2} \lesssim& \int_0^{t_*} \|u-u_0\|_{L^2} \|\bar a\|_{\dot W^{1,\infty}} + \mu\|\Lambda^{2\alpha}\bar a\|_{L^2} \\
\lesssim& t_* \|u-u_0\|_{L^2} \|\bar a\|_{\dot W^{1,\infty}} + \mu t_*\|\Lambda^{2\alpha}\bar a\|_{L^2} \\
\lesssim& \lambda^{\beta-4}\lambda^{2-\beta}N^{4-2\beta}(\lambda N) \lambda^{1-\beta}N^{-\beta} + \lambda^{\beta-4} (\lambda N)^{2\alpha} \lambda^{-\beta}N^{-\beta} \\
=& \lambda^{-\beta}N^{-\beta-1}N^{6-2\beta} + \lambda^{-\beta}\lambda^{\beta+2\alpha-4}N^{2\alpha-\beta}.
\end{align*}
Set $\lambda^{\beta+2\alpha-4}=N^{-100}$ and assume $\beta > 3$, we have
\[
\|A\|_{L^2} \le C \lambda^{-\beta}N^{-\beta-1} N^{-r}
\]
for small $r=2\beta-6>0$. Therefore, taking $N$ large such that $CN^{-r} < M$, the bootstrap is closed in $L^2$.

\noindent$\mathbf{\dot H^1}$ \textbf{case:}
\[
\|A\|_{\dot{H}^1} \lesssim \int_0^t  \left(\sum_{m'=0}^1\|u(\tau)-u_0\|_{\dot{H}^{m'}}\|\bar{a}(\tau)\|_{\dot{W}^{2-m',\infty}} + \mu\|\Lambda^{2\alpha+1} \bar a(\tau)\|_{L^2} \right) e^{\int_\tau^t\|u(\tau')\|_{\dot{W}^{1,\infty}}d\tau'}d\tau
\]
is similar to $L^2$ case, except a sum $\sum_{m'=0}^1\|u(\tau)-u_0\|_{\dot{H}^{m'}}\|\bar{a}(\tau)\|_{\dot{W}^{2-m',\infty}}$. However, note both $u-u_0$ and $\bar a$ have the same gaining factor $\lambda N$, thus it suffices to study just one term in the sum. In other words,
\[
\|A\|_{\dot{H}^1} \lesssim \int_0^t  \left(\|u(\tau)-u_0\|_{L^2}\|\bar{a}(\tau)\|_{\dot{W}^{2,\infty}} + \mu\|\Lambda^{2\alpha+1} \bar a(\tau)\|_{L^2} \right) e^{\int_\tau^t\|u(\tau')\|_{\dot{W}^{1,\infty}}d\tau'}d\tau.
\]
Then check the boundedness of the exponential follows from \eqref{eq-exp-bdd}, and similarly to in $L^2$ case we have
\begin{align*}
\|u-u_0\|_{L^2} \lesssim& t_* \|\nabla^\perp(\nabla^\perp \bar a \cdot \nabla \Delta \bar a)\|_{L^2} + t_* \|\nabla^\perp(\nabla^\perp A \cdot \nabla \Delta \bar a)\|_{L^2}
+ t_* \|\nabla^\perp(\nabla^\perp a \cdot \nabla \Delta A)\|_{L^2} \\
\lesssim& \lambda^{\beta-4}\lambda^{5-2\beta}N^{4-2\beta} + \lambda^{\beta-4}\lambda^{6-2\beta}N^{4-2\beta} \lesssim \lambda^{2-\beta}N^{4-2\beta}.
\end{align*}
Then,
\begin{align*}
\|A\|_{\dot H^1} \lesssim& \int_0^{t_*} \|u-u_0\|_{L^2} \|\bar a\|_{\dot W^{2,\infty}} + \mu\|\Lambda^{2\alpha+1}\bar a\|_{L^2} \\
\lesssim& t_* \|u-u_0\|_{L^2} \|\bar a\|_{\dot W^{2,\infty}} + \mu t_*\|\Lambda^{2\alpha+1}\bar a\|_{L^2} \\
\lesssim& \lambda^{\beta-4}\lambda^{2-\beta}N^{4-2\beta}(\lambda N)^2 \lambda^{1-\beta}N^{-\beta} + \lambda^{\beta-4} (\lambda N)^{2\alpha+1} \lambda^{-\beta}N^{-\beta} \\
=& \lambda^{1-\beta}N^{1-\beta-1}N^{6-2\beta} + \lambda^{1-\beta}\lambda^{\beta+2\alpha-4}N^{1-\beta+2\alpha}.
\end{align*}
Seting $\lambda^{\beta+2\alpha-4}=N^{-100}$ and assuming $\beta > 3$, yields
\[
\|A\|_{\dot H^1} \le C \lambda^{1-\beta}N^{1-\beta-1} N^{-r}
\]
Thus we close the bootstrap in $\dot H^1$ similarly as in $L^2$ by taking $N$ large.

\noindent$\mathbf{\dot H^m}$ \textbf{case:}
\begin{align*}
\|A\|_{\dot{H}^m} \lesssim \int_0^t&  \big(\sum_{m'=2}^m\|u\|_{\dot{W}^{m',\infty}}\|A\|_{\dot{H}^{m-m'+1}}+\sum_{m'=0}^m\|u(\tau)-u_0\|_{\dot{H}^{m'}}\|\bar{a}(\tau)\|_{\dot{W}^{m-m'+1,\infty}} \\
&+ \mu\|\Lambda^{2\alpha+m} \bar a(\tau)\|_{L^2} \big)e^{\int_\tau^t\|u(\tau')\|_{\dot{W}^{1,\infty}}d\tau'}d\tau.
\end{align*}
The proof is by induction, assume the bootstrap is closed for $s\le m-1$ and prove the bootstrap is closed in $\dot{H}^m$. Then, an interpolation between integers finishes the proof of the lemma. Note that all $u$, $u-u_0$, $A$, and $\bar a$ have gaining factor $\lambda N$. Thus, it suffices to write
\begin{align*}
\|A\|_{\dot{H}^m} \lesssim \int_0^t&  \big(\|u\|_{\dot{W}^{2,\infty}}\|A\|_{\dot{H}^{m-1}}+\|u(\tau)-u_0\|_{L^2}\|\bar{a}(\tau)\|_{\dot{W}^{m+1,\infty}} \\
&+ \mu\|\Lambda^{2\alpha+m} \bar a(\tau)\|_{L^2} \big)e^{\int_\tau^t\|u(\tau')\|_{\dot{W}^{1,\infty}}d\tau'}d\tau.
\end{align*}
Exponential is bounded by \eqref{eq-exp-bdd}, thus
\[
\|A\|_{\dot{H}^m} \lesssim \int_0^t \|u\|_{\dot{W}^{2,\infty}}\|A\|_{\dot{H}^{m-1}}+\|u(\tau)-u_0\|_{L^2}\|\bar{a}(\tau)\|_{\dot{W}^{m+1,\infty}} + \mu\|\Lambda^{2\alpha+m} \bar a(\tau)\|_{L^2} d\tau.
\]
Notice that, similarly as before, the last two terms above could be estimated as
\begin{equation} \label{eq-high-2-term}
\begin{split}
&\int_0^{t_*} \|u-u_0\|_{L^2} \|\bar a\|_{\dot W^{m+1,\infty}} + \mu\|\Lambda^{2\alpha+1}\bar a\|_{L^2} \\
\lesssim& t_* \|u-u_0\|_{L^2} \|\bar a\|_{\dot W^{2,\infty}} + \mu t_*\|\Lambda^{2\alpha+m}\bar a\|_{L^2} \\
\lesssim& \lambda^{\beta-4}\lambda^{2-\beta}N^{4-2\beta}(\lambda N)^{m+1} \lambda^{1-\beta}N^{-\beta} + \lambda^{\beta-4} (\lambda N)^{2\alpha+m} \lambda^{-\beta}N^{-\beta} \\
=& \lambda^{m-\beta}N^{m-\beta-1}N^{6-2\beta} + \lambda^{m-\beta}\lambda^{\beta+2\alpha-4}N^{m-\beta+2\alpha} \\
\lesssim& \lambda^{m-\beta}N^{m-\beta-1}N^{-r},
\end{split}
\end{equation}
and the first term has fewer powers on $N$ from induction assumption
\begin{align*}
\int_0^t \|u\|_{\dot{W}^{2,\infty}}\|A\|_{\dot{H}^{m-1}} \le& t_*\|u\|_{\dot{W}^{2,\infty}}\|A\|_{\dot{H}^{m-1}}\\
\le&  t_*(\lambda N)\|u\|_{\dot{W}^{1,\infty}}\|A\|_{\dot{H}^{m-1}}\\
\le& C_{u} (\lambda N) \epsilon^{100-2(m-1)}M \lambda^{m-1-\beta}N^{m-1-\beta-1} \\
\le&C_u\epsilon^{100-2(m-1)}M\lambda^{m-\beta}N^{m-\beta-1},
\end{align*}
where in the third line we used \eqref{eq-exp-bdd} and $C_u$ also comes from \eqref{eq-exp-bdd}:
\[
t_*\|u\|_{\dot W^{1,\infty}} \le \tilde C N^{6-2\beta} + t_*\|u_0\|_{\dot W^{1,\infty}} \le C_u.
\]
We need to show $C_u\epsilon^{100-2(m-1)} M < \epsilon^{100-2m}M$, then our proof of $\dot H^m$ case is finished combined with \eqref{eq-high-2-term} and large $N$. Note the first term in above equation goes to zero as $N \to \infty$ so it does not contribute to $C_u$. In other words,
\[
C_u = t_*\|u_0\|_{\dot W^{1,\infty}} = \epsilon^{-2}\lambda^{\beta-4}\epsilon\lambda^{4-\beta}= \epsilon^{-1}
\]
so
\[
C_u\epsilon^{100-2(m-1)} M = \epsilon^{100-2m+1} M< \epsilon^{100-2m}M
\]
for small $\epsilon<1$. For a non-integer $s\ge 0$, interpolation of the above proves \eqref{eq-est-A}, which concludes the proof of lemma.

\end{proof}

\section{Proof of Theorem \ref{thm:main}}
Combining everything above, assuming $3<\beta<4-2\alpha$, for any small $\epsilon>0$,
\[
\|a_0\|_{H^\beta} + \|b_0\|_{H^{\beta-1}} =\|a_0\|_{H^\beta} + \|u_0\|_{H^{\beta-2}} \lesssim \epsilon
\]
by Lemma \ref{est-u0}. By Lemma \ref{le-a-osc} and Lemma \ref{small-ubar},
\[
\|\bar a(t_*)\|_{\dot H^\beta} > \epsilon^{1-\beta}, \qquad \|\bar u(t)\|_{H^{\beta-2}} \lesssim \|u_0\|_{H^{\beta-2}} \lesssim \epsilon, \qquad \forall t\in[0,t_*].
\]
Then, by Lemma \ref{le-perturbation}, choosing large $\lambda$ and $N$, the difference $\bar a -a$ is under control and $u$ is small:
\[
\| a(t_*)\|_{\dot H^\beta} \gtrsim \epsilon^{1-\beta}, \qquad \| u(t)\|_{H^{\beta-2}} \lesssim \| u_0\|_{H^{\beta-2}} \lesssim \epsilon, \qquad \forall t\in[0,t_*]
\]
where all $\lesssim$ and $\gtrsim$ are independent of $\lambda$ and $N$. For any large $\Lambda>0$, take a smaller $\epsilon$ such that $\epsilon^{1-\beta}>\Lambda$. This concludes the proof of Theorem \ref{thm:main}.

\section{References}
\nocite{*}
\printbibliography[heading=none]

\end{document}